\documentclass[11pt, english, a4paper]{amsart}

\usepackage{graphicx}

\usepackage{amsmath}
\usepackage{amssymb}
\usepackage[utf8]{inputenc}
\usepackage[colorlinks=true, citecolor=blue, linkcolor=blue, urlcolor=blue]{hyperref}
\usepackage[T1]{fontenc}
\usepackage{amsfonts}
\usepackage{url}
\usepackage{mathpazo}  

\usepackage{mathtools}

\usepackage{graphicx}
\usepackage{verbatim}
\usepackage{enumerate}	
\usepackage[english]{babel}

\usepackage[shortlabels]{enumitem}

\usepackage{todonotes}

\usepackage[pdf]{graphviz}

\usepackage[margin=2.5cm]{geometry}

\theoremstyle{definition}
\newtheorem{definition}{Definition}
\newtheorem{example}[definition]{Example}

\newtheorem{remark}[definition]{Remark}

\theoremstyle{plain}
\newtheorem{lemma}[definition]{Lemma}
\newtheorem{proposition}[definition]{Proposition}
\newtheorem{theorem}[definition]{Theorem}
\newtheorem{corollary}[definition]{Corollary}

\makeatletter
\newcommand{\tpmod}[1]{{\@displayfalse\pmod{#1}}}
\makeatother

\newcommand\Ni[1][S]{\mathfrak{I}_{0}(#1)}
\newcommand{\Os}[1][S]{\mathcal{O}(#1)}
\newcommand{\Ap}{\operatorname{Ap}}
\newcommand{\Cl}{\mathcal{C}\!\ell}

\makeatletter
\def\@bignumber#1#2{%
  \ifx#2\end
    #1\let\next\@gobble
  \else
    #1\hspace{0pt plus 1pt}\let\next\@bignumber
  \fi
  \next#2}
\newcommand{\bignumber}[1]{\@bignumber#1\end}
\makeatother

\title[Poset normalized ideals]{The poset of normalized ideals of numerical semigroups with multiplicity three}
\author{S. Bonzio}
\address{Departament of Mathematics and Computer Science, University of Cagliari, via Ospedale 72, 09124 Cagliari, Italy}
\email{stefano.bonzio@unica.it}

\author{P. A. García-Sánchez}
\address{Departamento de \'Algebra and IMAG, Universidad de Granada, E-18071 Granada, Espa\~na}
\email{pedro@ugr.es}

\keywords{ideal class monoid, numerical semigroup, poset, lattice}
\subjclass[2020]{20M14, 20M12}

\begin{document}

\begin{abstract}
    We study the poset of normalized ideals of a numerical semigroup with multiplicity three. We show that this poset is always a lattice, and that two different numerical semigroups with multiplicity three have non-isomorphic posets of normalized ideals.
\end{abstract}

\maketitle

\section{Introduction}

Let $S$ be a numerical semigroup, and let $I,J$ be two ideals of $S$. We write $I\sim J$ if there exists $z\in \mathbb{Z}$ (the set of integers) such that $I=z+J$. This binary relation is an equivalence relation. Let us denote by $[I]$ the equivalence class of the ideal $I$. We can define the sum of two ideal classes as $[I]+[J]=[I+J]$, where, as usual, $I+J$ denotes the setwise addition of $I$ and $J$ (or Minkowski sum). If we denote by $\Cl(S)$ the quotient of the set of ideals of $S$ modulo $\sim$, then $(\Cl(S),+)$ is a monoid, known as the ideal class monoid of $S$; its identity element is $[S]$, the equivalence class of $S$. The ideal class monoid of a numerical semigroup was introduced in \cite{b-k}, where some basic properties and its relation to antichains of gaps of the numerical semigroup were given.

It is well known that $(\Cl(S),+)$ is isomorphic to $(\Ni,+)$, where $\Ni$ denotes the set of ideals $I$ of $S$ such that $\min(I)=0$ (normalized ideals). The isomorphism is precisely $[I]\mapsto -\min(I)+I$. In \cite{apery-icm} the concept of Kunz coordinates of normalized ideals of a numerical semigroup were introduced, and it gave rise to new bounds for the cardinality of the ideal class monoid of a numerical semigroup, as well as some closed formulas for the intersection, union, and sum of ideals. 

In \cite{isom-icm}, it was shown that if $S$ and $T$ are semigroups such that $(\Cl(S),+)$ is isomorphic to $(\Cl(T),+)$, then $S$ and $T$ must be equal. So, one of the motivations to study the ideal class monoid, which was classifying numerical semigroups, was completed: the ideal class monoid of a numerical semigroup completely determines the numerical semigroup.

We can define on $\Ni$ the following relation: $I\preceq J$ if there exists $K\in \Ni$ with $I+K=J$. On $\Cl(S)$, this relation translates to the relation $[I]\preceq [J]$ if $[I]+[K]=[J]$ for some ideal $K$ of $S$. Clearly, the posets $(\Ni,\preceq)$ and $(\Cl(S),\preceq)$ are order isomorphic.

In \cite{apery-icm}, several properties and invariants of $S$ were derived from the shape of the Hasse diagram of the poset $(\Ni,\preceq)$, and the natural question of whether an order isomorphism between $(\Ni,\preceq)$ and $(\Ni[T],\preceq)$ (with $T$ another numerical semigroup), forces $S=T$ was proposed. In \cite{isom-icm}, it was proven that if the posets $(\Ni,\subseteq)$ and $(\Ni[T],\subseteq)$ are order isomorphic, then $S=T$. As we will see in Section~\ref{sec:inclusion-eq-preceq} in very few cases the order relations $\preceq$ and $\subseteq$ coincide.

The aim of this paper is to study in deep detail the poset $(\Ni,\preceq)$ in the case $S$ has multiplicity three, and give an affirmative answer to the poset isomorphism problem proposed in \cite[Question~6.2]{apery-icm}. To this end, we will see that $S$ is fully determined by the number of quarks in $(\Ni,\preceq)$ and their depths. We will extensively make use of the Kunz coordinates of the normalized ideals of a numerical semigroup. In doing so, we gain some knowledge on operations with ideals given by Kunz coordinates. Some auxiliary lemmas not restricted to multiplicity three are also presented.

\subsection*{Supplemental online material} Most of the results presented here took shape after observing many experiments carried out with the help of the \texttt{numericalsgps} \cite{numericalsgps} \texttt{GAP} \cite{gap} package (see \url{https://github.com/numerical-semigroups/ideal-class-monoid}). The computations related to the examples in this manuscript can be found in that repository.

\section{Preliminaries}

In this section, we recall some basic notions and results concerning numerical semigroups, ideals of numerical semigroups, posets and lattices.

\subsection{Numerical semigroups}
A \emph{numerical semigroup} is a co-finite submonoid of the monoid of non-negative integers under addition, denoted by $\mathbb{N}$ in this manuscript. The co-finite condition is equivalent to saying that the greatest common divisor of the elements of the semigroup is one. 
The least positive integer belonging to $S$ is known as the \emph{multiplicity} of $S$, denoted by $\operatorname{m}(S)$, that is, $\operatorname{m}(S)=\min(S^*)$, with $S^*=S\setminus\{0\}$.

Given $A\subseteq \mathbb{N}$, the smallest submonoid of $\mathbb{N}$ that contains $A$ is 
\[
\langle A\rangle = \left\{ \sum_{i=1}^n a_i : n\in \mathbb{N}, a_1,\dots,a_n\in A\right\}.
\]
Clearly, if $S$ is a numerical semigroup, then $\langle S\rangle =S$. If $A\subseteq S$ is such that $\langle A\rangle =S$, then we say that $A$ is a \emph{set of generators} of $S$, or simply that $A$ \emph{generates} $S$. We say that $A$ is a \emph{minimal set of generators} of $S$ if no proper subset of $A$ generates $S$. It is well known that $S$ admits a unique minimal set of generators $S^*\setminus (S^*+S^*)$ (see for instance \cite[Corollary~2]{ns-app}), moreover this set cannot have two elements congruent modulo $\operatorname{m}(S)$, and so the cardinality of the minimal set of generators, known as the embedding dimension of $S$, is always smaller than the multiplicity of $S$. The elements of $S^*\setminus (S^*+S^*)$ are known as minimal generators of $S$. It follows easily that $s\in S$ is a minimal generator if and only if $S\setminus \{s\}$ is a numerical semigroup.

For a numerical semigroup $S$, the elements in $\mathbb{N}\setminus S$ are called \emph{gaps} of $S$. The cardinality of $\mathbb{N}\setminus S$ is the \emph{genus} of $S$, denoted by $\operatorname{g}(S)$. The largest gap of a numerical semigroup $S$ different from $\mathbb{N}$ is known as the \emph{Frobenius number} of $S$, that is, $\operatorname{F}(S)=\max(\mathbb{Z}\setminus S)$.

Associated to a numerical semigroup $S$ one can define the order induced by $S$ as $a\le_S b$ if $b-a\in S$, for any $a,b\in\mathbb{Z}$. Minimal generators of $S$ are precisely the elements in $S^*$ that are minimal with respect to $\le_S$.

The set of maximal elements of $\mathbb{Z}\setminus S$ are known as the \emph{pseudo-Frobenius numbers} of $S$, and this set is denoted by $\operatorname{PF}(S)$. Thus, $f\in \mathbb{Z}$ is in $\operatorname{PF}(S)$ if and only if $f\not\in S$ and $f+s\in S$ for all $s\in S^*$. The cardinality of $\operatorname{PF}(S)$ is called the (Cohen-Macaulay) \emph{type} of $S$, and it is denoted by $\operatorname{t}(S)$.

A numerical semigroup $S$ is \emph{irreducible} if it cannot be expressed as the intersection of two numerical semigroups properly containing it. Every irreducible numerical semigroup is either symmetric (odd Frobenius number) or pseudo-symmetric (even Frobenius number). Recall that a numerical semigroup $S$ is \emph{symmetric} if for any $x\in \mathbb{Z}\setminus S$, $\operatorname{F}(S)-x\in S$, and $S$ is \emph{pseudo-symmetric} if $\operatorname{F}(S)$ is even and for any $x\in \mathbb{Z}\setminus S$, $x\neq \operatorname{F}(S)/2$, $\operatorname{F}(S)-x\in S$. A numerical semigroup is symmetric if and only if its type is one, and it is pseudo-symmetric if and only if $\operatorname{PF}(S)=\{\operatorname{F}(S)/2,\operatorname{F}(S)\}$ (see \cite[Chapter~3]{ns} for more details). 

Let $S$ be a numerical semigroup. A gap $g$ of $S$ is called a \emph{special gap} if $S\cup\{g\}$ is a numerical semigroup. The set of special gaps is denoted by $\operatorname{SG}(S)$. It is well known that $\operatorname{SG}(S)=\{ g\in \operatorname{PF}(S) : 2g\in S\}$ \cite[Proposition~4.33]{ns}.

\subsection{Ideals of numerical semigroups}

A (relative) \emph{ideal} of $S$ is a set $I$ of integers such that $I+S=I$ and $z+I\subseteq S$ for some integer $z$. This last condition is equivalent to the existence of $\min(I)$. An ideal $I$ is said to be \emph{normalized} if $\min(I)=0$.

The union of two ideals of a numerical semigroup is again an ideal, and the same holds for the intersection. Addition of ideals can be defined as follows. If $I$ and $J$ are ideals of $S$, then $I+J=\{i+j : i\in I, j\in J\}$ is also an ideal of $S$ (see \cite[Chapter~3]{ns-app} for the basic properties of ideals of numerical semigroups). The set of ideals of $S$ under this operation is a monoid, and its identity element is $S$.

If $I$ is an ideal of a numerical semigroup $S$, then $I+S=I$, and so $I$ can be expressed as $I=X+S=\{x+s : x \in X, s\in S\}$ for some subset $X$ of $I$. Such a set is known as a \emph{generating set} of $I$. Notice that $I=X+S$, with $X=\operatorname{Minimals}_{\le_S}(I)$, and that every generating set of $I$ must contain $X$. The set $\operatorname{Minimals}_{\le_S}(I)$ is known as the \emph{minimal generating set} of $S$, and its elements are called minimal generators of $I$. Observe that there cannot be two different minimal generators congruent modulo the multiplicity of $S$. In particular, a minimal generating set of an ideal $I$ of $S$ has at most cardinality $\operatorname{m}(S)$. 

If $X$ is a finite set of integers, then $X+S$ is an ideal of $S$. If $X=\{x\}$ for some $x\in \mathbb{Z}$, then we write $x+S$ instead of $\{x\}+S$.

Let $I\in \Ni$. Then, $x\in I\setminus\{0\}$ is a minimal generator of $I$ if and only if $I\setminus\{x\}\in \Ni$ (see \cite[Lemma~9]{isom-icm}).

\subsection{Posets and lattices}

A \emph{poset} $(P,\leq)$ is a set equipped with a (partial) order relation ($\leq$). Given two elements $a,b$ in a poset $(P,\leq)$, $b$ is a \emph{cover} of $a$ if $a < b$ (that is $a\leq b$ and $a\neq b$) and for every $c$ such that $a\leq c\leq b$, then either $c = a$ or $c = b$. 

A poset $(P,\leq)$ is a \emph{lattice} if, for every pair of elements $x,y\in P$ there exist the infimum and the supremum of $\{x,y\}$; in such a case they are (binary) operations on the set $P$ usually denoted by $x\wedge y$ and $x\vee y$, and referred to as \emph{meet} and \emph{join}, respectively. A lattice can be equivalently defined as a set equipped with two idempotent, associative, commutative and absorptive operations $\wedge $ and $\vee$. In a lattice $L$, meets and joins of an arbitrary subset $X\subseteq L$ do not necessarily exist; if they do for every subset $X\subseteq L$, then $(L,\wedge,\vee)$ is a \emph{complete} lattice. The supremum and infimum of an arbitrary set $X$ (if they exist) are sometimes indicated with $\bigwedge X$ and $\bigvee X$, respectively.
Obviously, every finite lattice is a complete lattice. 

A \emph{meet semilattice} $(S,\wedge)$ is a set $S$ with an associative, commutative and idempotent operation $\wedge$. Every meet semilattice induces a partial order relation defined as: $x\leq y$ if and only if $x\wedge y = x$. If the order $\leq$ has a maximum element, denoted $1$, then $(S,\wedge)$ is called a meet semilattice with one ($1$ is, equivalently, the neutral element for the operation $\wedge$). Similarly, $(S,\vee)$ is a \emph{join semilattice} if $\vee $ is an associative, commutative and idempotent operation on $S$. In this case, $(S,\vee)$ induces a partial order relation defined as $x\le y$ if  $x\vee y = y$, for all $x,y\in S$. For the case the induced order $\leq$ has a minimum element, which we denote by $0$, then $(S,\vee)$ is a join semilattice with zero ($0$ is the neutral element for $\vee$). In a join (meet, respectively) semilattice the element $x\vee y$ ($x\wedge y$, respectively) is the supremum (infimum, respectively) of the set $\{x,y\}$ with respect to the induced order $\leq$.  
Given a poset $(P,\leq)$ and $X\subseteq P$, we indicate by ${\uparrow}X$ and ${\downarrow}X$ the set of upper and lower bounds, respectively, of $X$, namely ${\uparrow}X =\{a\in P  :  x\leq a, \text{ for every } x\in X\}$ and ${\downarrow}X =\{a\in P  :  a\leq x, \text{ for every } x\in X\}$. If $X = \{x\}$, we will write ${\uparrow}x$ and ${\downarrow}x$ instead of ${\uparrow}\{x\}$ and ${\downarrow}\{x\}$. 

In a poset $(P,\leq)$ we will say that an element $x\in P$ has a \emph{unique cover} if the poset $({({\uparrow}x) \setminus\{x\}}, \leq)$ has a minimum; we will denote this minimum as $x^{c}$.

Every finite meet or join semilattice with one or zero, respectively, is indeed a lattice, as recalled in the following well known result in order theory. 
\begin{theorem}\cite[Theorem 2.4]{Nation}\label{th: Nation Th: 2.4}
Let $(S,\vee)$ be a finite join semilattice with zero. Then, $S$ is a lattice with the meet operation defined by 
\[
x\wedge y = \bigvee ({\downarrow} x \cap {\downarrow} y).
\]
\end{theorem}

Recall that, in a poset $(P,\leq)$ two elements $x,y\in P$ are \emph{incomparable} (with respect to $\leq$) if $x\nleq y$ and $y\nleq x$.

\begin{lemma}\label{lemma: poset with unique cover}
Let $(P,\leq)$ be a poset and $x\in P$ an element having a unique cover $x^{c}$. Then: 
\begin{enumerate}
\item for any $y$ that is incomparable with $x$, it holds that ${\uparrow}\{x,y\} = {\uparrow}\{x^{c},y\}$; 
\item for any $y\nleq x$ and $y\leq x^{c}$, $x\vee y$ exists, in particular $x\vee y = x^{c}$.
\end{enumerate}
\end{lemma}
\begin{proof}
(1) Let $x^{c}$ be the unique cover of the element $x\in P$ and let $y$ be incomparable with $x$. Suppose $a\in {\uparrow}\{x,y\}$, that is, $x\leq a$ and $y\leq a$. Observe that, since  $x$ and $y$ are incomparable, we deduce that $a\neq x$. Hence, $a\in ({\uparrow} x)\setminus \{x\}$, and so  $x^{c}\leq a$. 
Thus, $a\in {\uparrow}\{x^{c},y\}$, showing that ${\uparrow}\{x,y\} \subseteq {\uparrow}\{x^{c},y\}$. The other inclusion is obvious as $x < x^{c}$.

\noindent
(2) Let $y \in P$ such that $y\nleq x$ and $y\leq x^{c}$. Then, $x^{c}$ in an upper bound of $\{x, y\}$. Let $z$ be an upper bound of $\{x, y\}$, that is, $x\leq z$ and $y\leq z$. The assumption $y\nleq x$ forces $x\neq z$. Thus, $z\in {\uparrow}x\setminus \{x\}$, and so $x^{c}\leq z$. This shows that $x\vee y = x^{c}$.
%
\end{proof}

\section{Normalized ideals of numerical semigroups}

Recall that a relative ideal $I$ of a numerical semigroup $S$ is a normalized ideal $I$ if $\min(I) = 0$. The set $\mathcal{I}_{0}(S)$ of normalized ideals of $S$ is always finite and forms a (complete) lattice under the operations of $\cap$ and $\cup$ (see \cite{apery-icm}). Moreover, as mentioned in the introduction, it can be turned into a poset upon considering the partial order relation: $I\preceq J$ if there exists $L\in\mathcal{I}_{0}(S)$ such that $I + L = J$. Observe that $I\preceq J$ implies $I\subseteq J$. An ideal $I$ is a \emph{quark} if it is minimal with respect to $\preceq$ in $\Ni\setminus\{S\}$; in other words, a quark is a cover of $S$ in the poset $(\Ni,\preceq)$.

Let $m$ be the multiplicity of $S$. If $I$ is an ideal of $S$, and $x\in I$, then $x+ks\in I$ for every non-negative integer $k$ and $s\in S$. It follows that the set $\Ap(I)=\{ x\in I : x-m\not \in I\}$ generates $I$ as an ideal, and it has precisely $m$ elements, one per each congruence class modulo $m$. The set $\Ap(I)$ is known as the \emph{Apéry set} of $I$ (with respect to $m$). Thus, $\Ap(I)=\{w_0,w_1,\dots,w_{m-1}\}$, where $w_i=\min(I\cap (i+m\mathbb{N}))$.

Given $a,b,n\in \mathbb{Z}$, with $n\neq 0$, we denote by $a\bmod n$ the remainder of the (Euclidean) division of $a$ by $n$ ($a\bmod n\in \{0,\dots,n_1\}$), and we write $a\equiv b \pmod n$ if $n$ divides $b-a$.

As $I\subseteq \mathbb{N}$, we deduce that $\Ap(I)\subseteq \mathbb{N}$, and consequently for every $i\in\{0,\dots,m-1\}$, $w_i=m x_i +i$ for some non-negative integer $x_i$. The tuple $(x_1,\dots,x_{m-1})$ is known as the \emph{Kunz coordinates} of $I$ (see \cite[Section~4]{apery-icm}). Notice that we are omitting $x_0$, which is equal to $0$, as $0=\min(I)$. 

In the sequel, we will write $I=(x_1,\dots,x_{m-1})_{\mathcal{K}}$ to denote that $(x_1,\dots,x_{m-1})$ are the Kunz coordinates of $I$.

Let $n\in \mathbb{N}$, and let $i= n\bmod m$.
Then, $n=km+i$ for some $k\in\mathbb{N}$. We know that $w_i$ is the minimum element in $I$ congruent with $i$ modulo $m$. Hence, $n\in I$ if and only if $n\ge w_i$, or equivalently $k\ge x_i$.  Hence, for an integer $n=k m+i$,
\begin{equation}\label{eq:kunz-membership}
   km+i\in (x_1,\dots,x_{m-1})_{\mathbb{K}}  \text{ if and only if } x_i\le k.
\end{equation}
In particular, $n\not\in I$ if and only if $k\in \{0,\dots,x_i-1\}$, and this holds for every congruence class modulo $m$. Thus, the number of non-negative integers not belonging to $I$ is
\begin{equation}\label{eq:genus-ideal-kunz}
    |\mathbb{N}\setminus (x_1,\dots,x_{m-1})_{\mathcal{K}}| = x_1+\dots+x_{m-1}.
\end{equation}

With this in mind it is easy to show that if $J$ is an ideal with Kunz coordinates $(y_1,\dots,y_{m-1})$, then
\begin{equation}\label{eq:inclusion-kunz-coord}
    (x_1,\dots,x_{m-1})_{\mathcal{K}}\subseteq (y_1,\dots,y_{m-1})_{\mathcal{K}} \text{ if and only if } (y_1,\dots,y_{m-1})\le (x_1,\ldots,x_{m-1})
\end{equation}
with respect to the usual partial order on $\mathbb{N}^{m-1}$, and
\begin{align*}
    (x_1,\dots,x_{m-1})_{\mathcal{K}}\cap (y_1,\dots,y_{m-1})_{\mathcal{K}} & =(\max(\{x_1,y_1\}),\dots,\max(\{x_{m-1},y_{m-1}\}))_{\mathcal{K}},\\
    (x_1,\dots,x_{m-1})_{\mathcal{K}}\cup (y_1,\dots,y_{m-1})_{\mathcal{K}} & =(\min(\{x_1,y_1\}),\dots,\min(\{x_{m-1},y_{m-1}\}))_{\mathcal{K}}.
\end{align*}
Addition requires more effort, but can be derived by translating \cite[Proposition~4.8]{apery-icm} to Kunz coordinates. If $I+J=(z_1,\dots,z_{m-1})_{\mathcal{K}}$, then for every $i\in \{1,\dots,m-1\}$
\begin{equation}\label{eq:kunz-coordinates-sum-general}
    z_i=\min(\{x_{i_1}+y_{i_2}+\lfloor \tfrac{i_1+i_2}m \rfloor : i_1,i_2\in \{0,\dots,m-1\}, i_1+i_2\equiv i \tpmod{m}\}).
\end{equation}
where $\lfloor q\rfloor=\max\{ z\in \mathbb{Z} : z\le q\}$ for every $q\in \mathbb{Q}$.

The idempotent elements in $(\Ni,+)$ are precisely the oversemigroups of $S$ \cite[Proposition~5.14]{apery-icm}, that is, numerical semigroups $T$ such that $S\subseteq T$. They will play a central role in Section~\ref{sec:lattice}. We denote by $\Os$ the set of oversemigroups of $S$.

\begin{remark}
  One may wonder when $\Os$ coincides with $\Ni$. This question was solved already in \cite[Proposition~4.4]{b-k}: $\Os=\Ni$ if and only if the multiplicity of $S$ is less than or equal to two. In this case, $(\Ni,\preceq)$ is a chain of length $\operatorname{g}(S)+1$.
\end{remark}

The following two technical lemmas will be handy later.

\begin{lemma}\label{lem:Os-inclusion-le}
  Let $S$ be a numerical semigroup, and let $I\in\Ni$, $J\in \Os$. Then, $I\subseteq J$ if and only if $I+J=J$.
  \end{lemma}
  \begin{proof}
    If $I+J=J$, then $I=I+0\subseteq I+J=J$. If  $I\subseteq J$, then $J=0+J\subseteq I+J\subseteq J+J=J$, and so $J=I+J$.
  \end{proof}
  
  \begin{lemma}\label{lem:join-with-idempotent}
    Let $S$ be a numerical semigroup, and let $I\in \Os$. For every $J\in \Ni$, $I\preceq J$ if and only if $J=I+J$. In particular, $I\vee J=I+J$ for every $J\in \Ni$.
  \end{lemma}
  \begin{proof}
    Suppose that $I\preceq J$, and so $I+K=J$ for some $K\in \Ni$. Then, $J=I+K=I+I+K=I+J$. The converse is trivial. 
    
    Now, suppose that $K\in \Ni$ is such that $I\preceq K$ and $J\preceq K$. Then, there exists $L$ for which $J+L=K$. We now use that $I+K=K$, and obtain that $I+J+L=I+K=K$, which means that $I+J\preceq K$. Also, $I\preceq I+J$ and $J\preceq I+J$, which proves that $I\vee J=I+J$.
  \end{proof}

As a consequence of these two lemmas, we get the following result.
\begin{corollary}
Let $S$ be a numerical semigroup, and let $I,J\in \Os$. The following are equivalent:
\begin{enumerate}
    \item $I\subseteq J$,
    \item $I+J=J$,
    \item $I\preceq J$.
\end{enumerate}
\end{corollary}

The following result will be used later, and describes the set of ideals of an oversemigroup $T$ of $S$ as the set of ideals in $S$ greater than or equal to $T$ with respect to $\preceq$.

\begin{lemma}\label{lem:oversemigroup-uparrow}
    Let $S$ be a numerical semigroup and let $T\in \Os$. Then, $\Ni[T]={\uparrow} T$ (in $\Ni$).
\end{lemma}
\begin{proof}
    If $I\in \Ni[T]$, then $I+T=I$ and consequently $T\preceq I$. 

    Let $I\in {\uparrow}T$. Then, $T\preceq I$ and so there exists $J\in \Ni$ such that $T+J=I$. Thus, $I+T=J+T+T=J+T=I$, and so $I\in \Ni[T]$.
\end{proof}

\subsection{Multiplicity three}

Let $S$ be a numerical semigroup with multiplicity three, with Kunz coordinates $S=(u,v)_{\mathcal{K}}$. Then, by \cite[Theorem~4.4]{apery-icm}, $(x,y)$ are the Kunz coordinates of a (normalized) ideal $I$ of $S$ if and only if
\begin{equation}\label{eq:kunz-cond}
\begin{cases}
  x\le u,\\
  y\le v,\\
  x+u\ge y,\\
  y+v+1\ge x.
\end{cases}
\end{equation}

Basically, the first two inequalities mean that $S\subseteq I$, while the last two inequalities translate to $I+S=I$.

For $I=(a,b)_{\mathcal{K}}$ and $J=(c,d)_{\mathcal{K}}$, \eqref{eq:kunz-coordinates-sum-general} particularizes to
\begin{equation}\label{eq:kunz-sum-ideals}
    I+J=(\min(\{a,c,b+d+1\}), \min(\{b,d,a+c\}))_{\mathcal{K}}.
\end{equation}

\begin{remark}\label{rem:two-cases-non-idempotent}
  As a particular instance of \eqref{eq:kunz-sum-ideals}, for $I=(a,b)_{\mathcal{K}}$, we obtain 
  \[I+I=(\min(\{a,2b+1\}), \min(\{b,2a\}))_{\mathcal{K}}.\] 
  Therefore,  $I+I\neq I$ if and only if either $2b+1<a$ or $2a<b$ (not both, since this would lead to $4b+2<2a<2b$, a contradiction).
\end{remark}

Let $S$ be a numerical semigroup with multiplicity three and Kunz coordinates $(u,v)$. Let $I=(x,y)_{\mathcal{K}}$ be an ideal of $S$. If $I$ is an idempotent ideal, then by Remark~\ref{rem:two-cases-non-idempotent}, $x\leq 2y +1$ and $y\leq 2x$. It is easily checked that for idempotents, \eqref{eq:kunz-cond} turns into: 
\begin{equation}\label{eq:kunz-cond-idempotentes}
\begin{cases}
  x\le u,\\
  y\le v,\\
  x\le 2y +1,\\
  y\le 2x.
\end{cases}
\end{equation}

Figure~\ref{fig:h-d-3} shows the Hasse diagram of $(\mathfrak{I}_0(\langle 3,13,17\rangle),\preceq)$. Ideals are represented and placed according to their Kunz coordinates. Idempotents are displayed in grey.

\begin{figure}
  \includegraphics[scale=0.5]{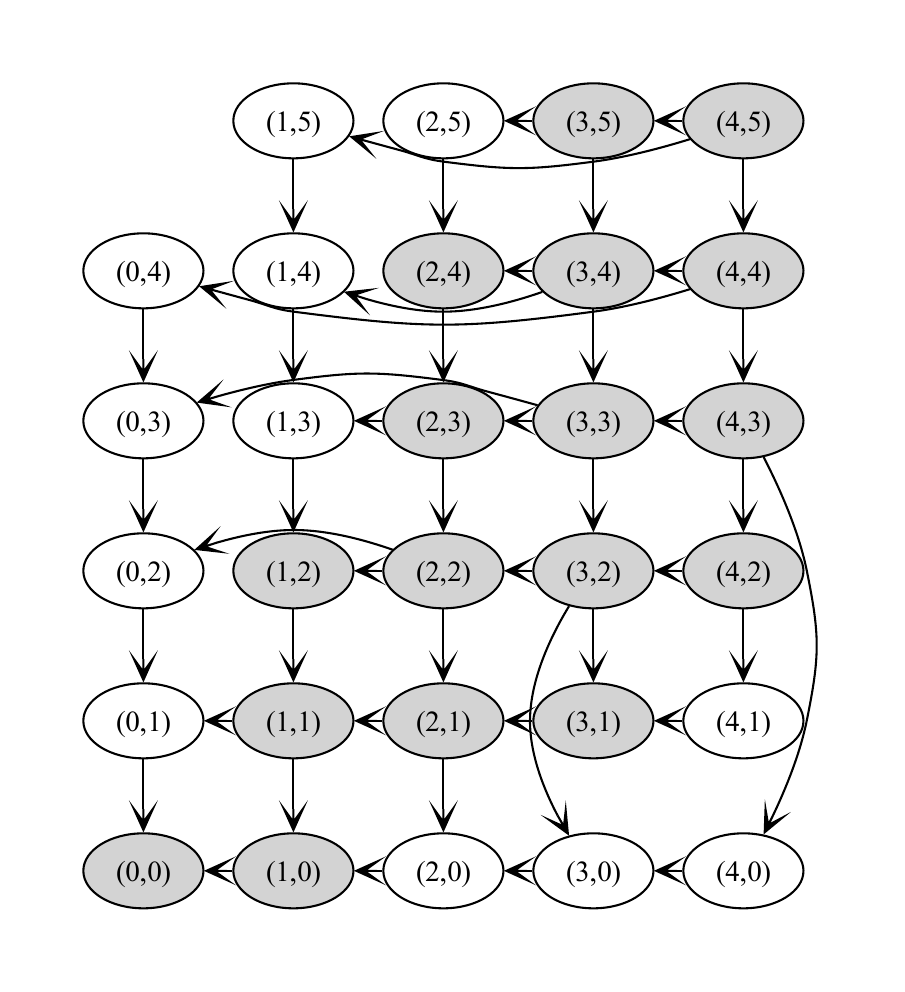}
\caption{Hasse diagram of $(\mathfrak{I}_0(\langle 3,13,17\rangle),\preceq)$; ideals are represented and placed according to their Kunz coordinates.}  
\label{fig:h-d-3}
\end{figure}

\begin{proposition}\label{prop:union-eq-sum-mult-3}
  Let $S$ be a numerical semigroup with multiplicity three. For all $I,J\in \Os$, $I+J=I\cup J$.
\end{proposition}
\begin{proof}
  Let $a_1,b_1,a_2,b_2\in \mathbb{N}$ such that $I=(a_1,b_1)_{\mathcal{K}}$ and $J=(a_2,b_2)_{\mathcal{K}}$. If $I\subseteq J$, then by Lemma~\ref{lem:Os-inclusion-le}, we have that $J=I+J$, which leads to $I\cup J=J=I+J$. The same argument works for $J\subseteq I$. Thus, we may suppose that $I$ and $J$ are incomparable with respect to set inclusion, which by \eqref{eq:inclusion-kunz-coord} means that $(a_1,b_1)$ and $(a_2,b_2)$ are incomparable in $(\mathbb{N}^2,\le)$. 

  Suppose that $a_1\le a_2$ and $b_1\ge b_2$. We know by \eqref{eq:kunz-coordinates-sum-general} that 
  \[
  I+J=(\min(\{a_1,b_1,a_2+b_2+1\}), \min(\{a_2,b_2,a_1+b_1\}))_{\mathcal{K}}.
  \]
  Also, $a_1\le a_2\le a_2+b_2+1$, 
  and so $\min(\{a_1,b_1,a_2+b_2+1\})=\min(\{a_1,b_1\})$. Similarly, $a_1+b_1\ge a_1+b_2\ge b_2$, and consequently $\min(\{a_2,b_2,a_1+b_1\})= \min(\{a_2,b_2\})$. Thus,
  \[
  I+J=(\min(\{a_1,b_1\}), \min(\{a_2,b_2\}))_{\mathcal{K}}=I\cup J.
  \]
  The case $a_1\ge a_2$ and $b_1\le b_2$ follows analogously.
\end{proof}

With this, we easily obtain the following consequence.

\begin{corollary}\label{cor: Os reticulo}
  Let $S$ be a numerical semigroup with multiplicity three. Then, $(\Os,\preceq)$ is a distributive lattice with $I\vee J=I\cup J$, $I\wedge J=I\cap J$, $\max_\preceq \Os=\mathbb{N}$, and $\min_\preceq \Os=S$.
\end{corollary}
\begin{proof}
  Let $K\in \Os$ such that $I\preceq K$ and $J\preceq K$. Then, $I\subseteq K$ and $J\subseteq K$, which yields $I\cup J\subseteq K$. By Lemma~\ref{lem:Os-inclusion-le}, $(I\cup J)+K=K$ and so $I\cup J\preceq K$. This proves that $I\cup J=I\vee J$.

  Analogously, one shows that $I\cap J=I\wedge J$. The other assertions are trivial.
\end{proof}

\section{When inclusion coincides with the order induced by addition}\label{sec:inclusion-eq-preceq}

For a numerical semigroup $S$ and  $I,J\in \Ni$, recall that $I\preceq J$ implies $I\subseteq J$. Thus, it is natural to ask under which circumstances the order relations $\preceq$ and $\subseteq$ coincide. 
  
\begin{example}
  The numerical semigroup with the least possible genus such that $\subseteq$ and $\preceq$ are not the same is $S=\langle 4,5,6,7\rangle$: $\{0,2\}+S\subseteq \{0,1,2\}+S$, while $\{0,2\}+S\not\preceq \{0,1,2\}+S$ (dashed line in Figure~\ref{fig:subset-not-preceq}).
\end{example}

\begin{figure}
    \centering
        \raisebox{2.6em}{%
        \includegraphics[scale=0.6]{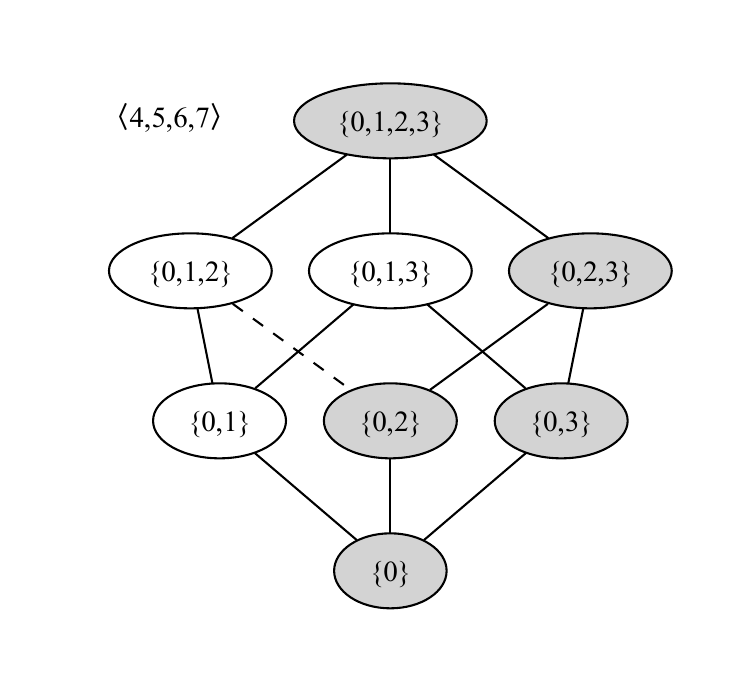}} \quad       \includegraphics[scale=0.6]{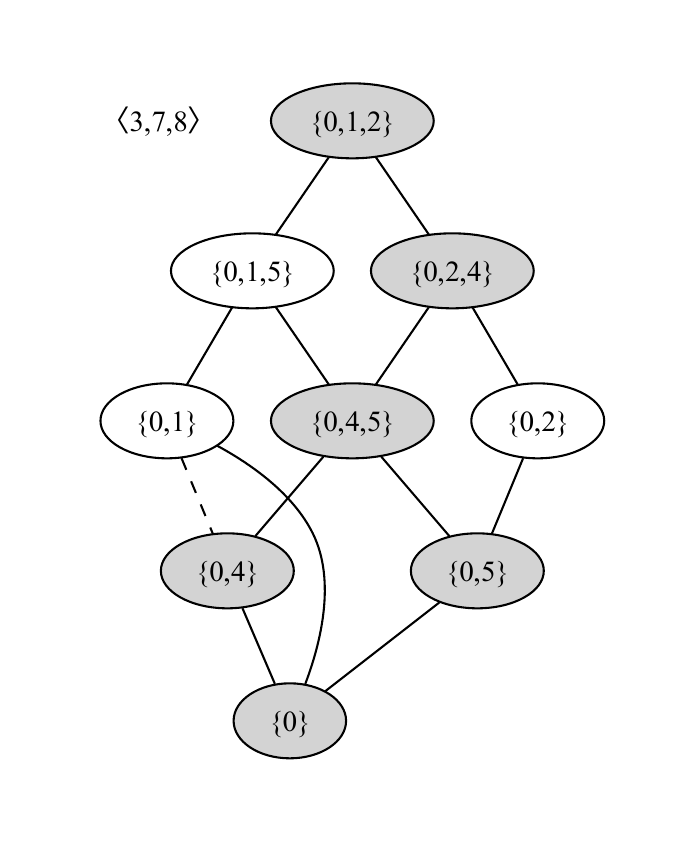}
    \caption{The posets $(\Ni[\langle 4,5,6,7\rangle],\preceq)$ and $(\Ni[\langle 3,7,8\rangle],\preceq)$, witnessing that the order relations $\subseteq$ and $\preceq$ are not equal in general.}
    \label{fig:subset-not-preceq}
\end{figure}

\begin{proposition}\label{prop:contained-preceq-coincide}
  Let $S$ be a numerical semigroup. Then, $\subseteq$ equals $\preceq$ in $\Ni$ if and only if 
  \[S\in \{ \langle 3,4\rangle, \langle 3,4,5\rangle, \langle 3,5\rangle, \langle 3,5,7\rangle \}\cup \{\langle 2,2k+1\rangle : k \in\mathbb{N}\}.\]
\end{proposition}
\begin{proof}
  Suppose that $\operatorname{m}(S)\ge 4$. Then, $\{0,2\}+S\subseteq \{0,1,2\}+S$. If there exists $I\in \Ni$ with $(\{0,2\}+S)+I=\{0,1,2\}+S$, then $1\in I$. But then, $1+2=3\in (\{0,2\}+S)+I$ and $3\not\in \{0,1,2\}+S$, a contradiction. Thus, $\subseteq$ equals $\preceq$ in $\Ni$ forces $\operatorname{m}(S)\le 3$. We already know that for $\operatorname{m}(S)\le 2$, we have that $\Os=\Ni$, and thus by Corollary~\ref{cor: Os reticulo}, $\subseteq$ equals $\preceq$.
  
  It remains to see what happens in the case $\operatorname{m}(S)=3$. Suppose that $4\not\in S$. Notice that $\{0,4\}+S\subseteq \{0,1\}+S$, since $4=1+3\in \{0,1\}+S$. If there exists $I$ with $(\{0,4\} + S)+I=\{0,1\}+S$, then $1\in I$. 
  This implies that $5 \in \{0,1\} + S $ (as $5 = 4 + 1\in(\{0,4\} + S) + I$).
  Hence, if $S$ has multiplicity three and $\subseteq$ equals $\preceq$, then $5\in S$, and so $\langle 3,5\rangle \subseteq S$. There are only two oversemigroups of $\langle 3,5\rangle$ with multiplicity three and not containing four, namely $\langle 3, 5, 7 \rangle$ and $\langle 3, 5 \rangle$ (this can be easily checked with the help of \cite{numericalsgps}): for all of them $\subseteq$ and $\preceq$ coincide in $\Ni$. The only case left is when $\langle 3,4\rangle \subseteq S$. This case yields two oversemigroups with multiplicity three,  $\langle 3,4\rangle$ and $\langle 3,4,5 \rangle$, both fulfilling that $\subseteq$ and $\preceq$ coincide in $\Ni$.  
\end{proof}

It follows from the above proposition that, for $S\in \{ \langle 3,4\rangle, \langle 3,4,5\rangle, \langle 3,5\rangle, \langle 3,5,7\rangle \}\cup \{\langle 2,2k+1\rangle : k \in\mathbb{N}\}$, $(\Ni,\preceq)$ is a lattice. However, as the next example shows, this does not cover all the cases for which $(\Ni,\preceq)$ is indeed a lattice.

\begin{example}
  One can check that for $S=\langle 3,7,8\rangle$, the poset $(\Ni,\preceq)$ is a lattice, however we already know that in $\Ni$ the orders $\preceq$ and $\subseteq$ are not the same in virtue of Proposition~\ref{prop:contained-preceq-coincide}. In fact, $\{0,4\}+S\subseteq \{0,1\}+S$, while $\{0,4\}+S\not\preceq \{0,1\}+S$ (see Figure~\ref{fig:subset-not-preceq}).
\end{example}

\begin{remark}
Casabella proves in \cite{laura} that if $S$ is a numerical semigroup such that the ideal class semigroup of $S$  and the ideal class semigroup of the semigroup algebra of $S$ ($K[S]$) are isomorphic, then $S$ has either multiplicity less than two or $S=\langle 3,4,5\rangle$ or $S=\langle 3,4\rangle$. By Proposition \ref{prop:contained-preceq-coincide}, this means that if $S$ and $K[S]$ have isomorphic ideal class monoids, then $\preceq$ and $\subseteq$ coincide in $\Ni$.   
\end{remark}

\section{The poset of normalized ideals of a numerical semigroup with multiplicity three is a lattice}\label{sec:lattice}

Next, we show that for any numerical semigroup $S$ with multiplicity three, $(\Ni,\preceq)$ is a lattice. First (and motivated by Lemma~\ref{lem:join-with-idempotent}), we will focus our attention on ideals that are not idempotent.

Recall that if $(a,b)_{\mathcal{K}}\in\Ni\setminus \Os$, then either $2b+1<a$ or $2b<b$ (Remark~\ref{rem:two-cases-non-idempotent}). Next, we study these two cases separately, showing that $(a,b)_{\mathcal{K}}$ has a unique cover.

\begin{lemma}\label{lem:m3-case-1}
Let $S$ be a numerical semigroup with multiplicity three and let $I=(a,b)_{\mathcal{K}}$ be an ideal of $S$ such that $2b+1<a$. Then, $J =(a-1,b)_{\mathcal{K}}$ is the unique cover of $I$ with respect to $\preceq$.
Moreover, $I+L=J+L$ for all $L\in \Ni$.

\end{lemma}
\begin{proof}
Take the pair $(a,b+t)$ with $t=(a-1)-(2b+1)$. By hypothesis, $2b+1\le (a-1)$, and so $t\in\mathbb{N}$. Let $(u,v)$ be the Kunz coordinates of $S$. Notice that: 
\begin{enumerate}
  \item $a\le u$ (first inequality if \eqref{eq:kunz-cond}, for $(x,y)=(a,b)$);
  \item $b+t=(a-1)-(b+1)\le v$ if and only if $a-1\le v+b+1$; but $a-1\le a\le b+v+1$ (last inequality of \eqref{eq:kunz-cond});
  \item observe that $b+t=(a-1)-(b+1)\le a+u$ which trivially holds (as $a\leq u$);
  \item $b+t+v+1\ge a$, since $b+t+v+1\ge b+v+1\ge a $ (last inequality in \eqref{eq:kunz-cond}).
\end{enumerate}
Thus, $(a,b+t)$ are the Kunz coordinates of a normalized ideal of $S$, say $L$; whence $L=(a,b+t)_{\mathcal{K}}$. Then,
\[
I+L=(\min\{a,2b+t+1\},\min\{b,b+t,2a\})_{\mathcal{K}}=(a-1,b)_{\mathcal{K}}=J.
\]
This proves $I\preceq J$. Notice that by \eqref{eq:genus-ideal-kunz}, $|J \setminus I|= (a+b)-(a+b-1)=1$ (as a matter of fact, $J\setminus I = \{3a+1\}$). Suppose that there exists $K$ with $I\precneqq K\precneqq J$. Then, $I\subsetneq K\subsetneq J$, and this would lead to $|J\setminus I|\ge 2$, a contradiction. Hence, $J$ covers $I$.

Next, we prove that if $I\precneqq L$ for some $L\in \Ni$, then $J\preceq L$. Write $L=(c,d)_{\mathcal{K}}$. As $I\prec L$, there exists $(x,y)_{K}$ fulfilling \eqref{eq:kunz-cond} such that $(a,b)_{\mathcal{K}}+(x,y)_{\mathcal{K}}=(c,d)_{\mathcal{K}}$, and $c=\min\{a,x,b+y+1\}$ and $d=\min\{b,y,a+x\}=\min\{b,y\}$ (by hypothesis $b<a$, and so $b< x+a$). We distinguish two cases depending on $x<a$ or $x\ge a$.
\begin{itemize}
  \item $x<a$. In this setting $(a-1,b)_{\mathcal{K}}+(x,y)_{\mathcal{K}}=(c,d)_{\mathcal{K}}$, since $\min\{a-1,x,b+y+1\}=\min\{x,b+y+1\}=\min\{a,x,b+y+1\}$.
  \item $x\ge a$. If $a\le b+y+1$, then $2b+1< a\le b+y+1$ and consequently $b\le y$. Under these conditions, $(a,b)_{\mathcal{K}}+(x,y)_{\mathcal{K}}=(a,b)_{\mathcal{K}}=(c,d)_{\mathcal{K}}$, which is impossible, as we are assuming that $I\neq L$. Thus, $a>b+y+1$, and so $a-1\ge b+y+1$. In this case, $(a-1,b)_{\mathcal{K}}+(x,y)_{\mathcal{K}}=(a,b)_{\mathcal{K}}+(x,y)_{\mathcal{K}}=(c,d)_{\mathcal{K}}$, since $\min\{a-1,x,b+y+1\}=b+y+1=\min\{a,x,b+y+1\}$.
\end{itemize}
This concludes the proof, since in all the possible cases $J\preceq L$. Notice that we have also shown that if $I+M=L$, then $J+M=L$, which proves the second part of the statement.
\end{proof}

\begin{lemma}\label{lem:m3-case-2}
  Let $S$ be a numerical semigroup with multiplicity three and let $I=(a,b)_{\mathcal{K}}$ be an ideal of $S$ such that $2a<b$. Then, $J =(a,b-1)_{\mathcal{K}}$ is the unique cover of $I$ with respect to $\preceq$.
Moreover, $I+L=J+L$ for all $L\in \Ni$. 
\end{lemma}
\begin{proof}
Since $2a<b$ then $2a\le b-1$. Set $t=b-1-2a\in \mathbb{N}$. Let us consider the pair $(a+t,b)$ and let us check that it fulfills the inequalities \eqref{eq:kunz-cond}, for $(u,v)_{S}$ a generic element in $S$. Observe that
  \begin{enumerate}
    \item $a+t\le u$ if and only if $b-1-a\le u$, which is equivalent to $b\le a+u+1$. By \eqref{eq:kunz-cond}, third inequality, we know that $b\le a+u$, and so $a+t\le u$ holds;
    \item $b\le v$, which holds by the second inequality of \eqref{eq:kunz-cond}, as $I=(a,b)_{\mathcal{K}}$ is an ideal;
    \item $a+t+u\ge a+u\ge b$, in light of the third inequality of \eqref{eq:kunz-cond} applied to $I$;
    \item $b+v+1\ge a+t$ if and only if $b+v+1\ge b-1-a$, equivalent to $v+1\ge -1-a$, which trivially holds.
  \end{enumerate}

  Thus, there exists $L\in \Ni$ such that $L=(a+t,b)_{\mathcal{K}}$. Moreover,
  \[
  I+L=(a+t,b)_{\mathcal{K}}+(a,b)_{\mathcal{K}}=(\min\{a,a+t,2b+1\}, \min\{b,b-1\})_{\mathcal{K}}=(a,b-1)_{\mathcal{K}}=J.
  \]
  In particular, $I\preceq J$, $J\setminus I=\{3b+2\}$ and so $J$ covers $I$ (arguing as in the proof of Lemma~\ref{lem:m3-case-1}).
  
  Now, we show that for every $L=(c,d)_{\mathcal{K}}\in \Ni$ with $I\precneqq L$, we have that $J\preceq L$. As in the proof of Lemma~\ref{lem:m3-case-1}, there exists $(x,y)$ fulfilling \eqref{eq:kunz-cond} such that $(a,b)_{\mathcal{K}}+(x,y)_{\mathcal{K}}=(c,d)_{\mathcal{K}}$. In particular, $c=\min\{a,x,b+y+1\}=\min\{a,x\}$ (by hypothesis $a<b$, and so $a<b+y+1$) and $d=\min\{b,y,a+x\}$. We distinguish between $y<b$ and $y\ge b$.
  \begin{itemize}
    \item $y<b$. In this setting, $\min\{b-1,y,a+x\}=\min\{y,a+x\}=\min\{b,y,a+x\}$, and so  $(a,b-1)_{\mathcal{K}}+(x,y)_{\mathcal{K}}=(c,d)_{\mathcal{K}}$.
    \item $y\ge b$. If $b\le a+x$, then $2a<b\le a+x$, and $a<x$. In this case, $(a,b)_{\mathcal{K}}+(x,y)_{\mathcal{K}}=(a,b)_{\mathcal{K}}=(c,d)_{\mathcal{K}}$, contradicting that $I\neq L$. Hence, $b>a+x$, which leads to $b-1\ge a+x$. It follows that $(a,b-1)_{\mathcal{K}}+(x,y)_{\mathcal{K}}=(a,b)_{\mathcal{K}}+(x,y)_{\mathcal{K}}=(c,d)_{\mathcal{K}}$, because under the standing hypothesis, $\min\{b-1,y,a+x\}=a+x=\min\{b,y,a+x\}$.
  \end{itemize}
  In any case, $J\preceq L$ and this concludes the proof.
\end{proof}

By combining the two previous lemmas, we obtain the following.

\begin{proposition}\label{prop:m3-non-idemp-one-cover}
  Let $S$ be a numerical semigroup with multiplicity three, and let $I\in \Ni\setminus \Os$. Then, $I$ has a unique cover with respect to $\preceq$, which we denote by $I^c$. Moreover, $I+J=I^c+J$ for all $J\in \Ni$.
\end{proposition}

Non-idempotent ideals have a unique cover; moreover, a stronger fact holds, namely we can also prove that two different non-idempotent ideals cannot share the same cover.

\begin{lemma}\label{lem:m3-non-indemptent-same-cover}
  Let $S$ be a numerical semigroup with multiplicity three, and let $I,J\in \Ni\setminus \Os$. If $I^c=J^c$, then $I=J$.
\end{lemma}
\begin{proof}
  Let $(a,b)$ and $(c,d)$ be the Kunz coordinates of $I$ and $J$, respectively. Recall that by Remark~\ref{rem:two-cases-non-idempotent} either $2b+1<a$ or $2a<b$ (not both) and either $2d+1<c$ or $2c<d$. 
  \begin{itemize}
    \item If $2b+1<a$, then by Lemma~\ref{lem:m3-case-1}, $I^c=(a-1,b)_{\mathcal{K}}$. If $2d+1<c$, then $J^c=(c-1,d)_{\mathcal{K}}$. Thus, if $I^c=J^c$, we deduce that $a=c$ and $b=d$, and so $I=J$. Suppose that $2c<d$. Then, by Lemma~\ref{lem:m3-case-2}, $J^c=(c,d-1)_{\mathcal{K}}$; whence $a-1=c$ and $d-1=b$. As $2b+1<a$, we deduce $2d-2+1<c+1$, or equivalently, $2d<c+2$, and so $4d<2c+4<d+4$, which implies $d\leq 1$. If $d = 1$, then $c = 0$ (as $2c<d$), hence $a = 1$ and $b = 0$, which is impossible as $2b+1<a$. Similarly, also $d=0$ is impossible as $2c < d$. Therefore, this can never be the case.
    \item If $2a<b$ and $2c<d$, then by Lemma~\ref{lem:m3-case-2}, $I^c=(a,b-1)_{\mathcal{K}}=(c,d_1)_{\mathcal{K}}=J^c$, which clearly leads to $I=J$. If $2d+1<c$, then we can argue as in the previous case. 
  \end{itemize}
  This covers all possible cases, and concludes the proof.
\end{proof}

\begin{lemma}\label{lem:m3-join-non-idempotent}
  Let $S$ be a numerical semigroup with multiplicity three and $I=\Ni\setminus \Os$. If $J\npreceq I$ and $J\precneqq I^{c}$, then $I\vee J=I+J$.
\end{lemma}
\begin{proof}
Observe that the fact that $I\vee J$ exists follows from Lemma~\ref{lemma: poset with unique cover}-(2).
As $J\precneqq I^c$, there exists $J_1,\dots,J_t\in \Ni$ such that $J_1=J$, $J_t=I^c$ and $J_{i+1}$ covers $J_i$ for all $i\in \{1,\dots,t-1\}$. As $J\neq I^c$, we have that $t>1$.

If $J_i\not\in \Os$ for all $i$, then we have that $J_{t-1}^c=I^c$. By Lemma~\ref{lem:m3-non-indemptent-same-cover}, $J_{t-1}=I$, but then $J\preceq J_{t-1}=I$, contradicting $J\not\preceq I$. 

Therefore, $J_i\in \Os$ for some $i\in \{1,\dots,t\}$. Let $i$ be minimum with this property. By Lemma~\ref{lem:join-with-idempotent}, $J_i\vee I^c=J_i+I^c=I^c$. Also, as $J\not\preceq I$, we deduce that $J_i\not\preceq I$. By Lemma \ref{lemma: poset with unique cover}-(2), $J\vee I=J_i\vee I=I^c$. Notice also that $J_j\in \Ni\setminus \Os$ for all $j<i$. By applying several times Proposition~\ref{prop:m3-non-idemp-one-cover}, we obtain $J+I^c=J_1+I^c=\dots=J_i+I^c$. Thus, $J+I^c=I^c=J\vee I$. By using again Proposition~\ref{prop:m3-non-idemp-one-cover}, but now with $I$, $J+I^c=J+I$ and we conclude that $I\vee J=I^c=I+J$.
\end{proof}

With all these technical lemmas we are ready to prove that $(\Ni,\preceq)$ is always a lattice for $S$ a numerical semigroup with multiplicity three. Also, in this case, we can give an explicit formula for the description of two ideals.
 
\begin{theorem}\label{th: m=3 is a lattice}
  Let $S$ be a numerical semigroup with multiplicity three. Then, $(\Ni,\preceq)$ is a lattice. Moreover, for  $I,J\in \Ni$,
  \[
    I\vee J=\begin{cases}
      I\cup J, \text{ if } I \text{ and } J \text{ are comparable},\\
      I+J,\text{ otherwise,}
    \end{cases}
  \]  
  and 
   \[
    I\wedge J=\begin{cases}
      I\cap J, \text{ if } I \text{ and } J \text{ are comparable},\\
      \bigvee ({\downarrow}I \cap{\downarrow}J),\text{ otherwise.}
    \end{cases}
  \]
\end{theorem}

\begin{proof}
  Let $I,J\in \Ni$. We show that there exists $I\vee J$. If $I\preceq J$, then $I\vee J=J=I\cup J$. Analogously, if $J\preceq I$, then $I\vee J=I=I\cup J$. So, we can suppose that $I$ and $J$ are not comparable with respect to $\preceq$. 

  If $I\in \Os$, then by Lemma~\ref{lem:join-with-idempotent}, $I\vee J=I+J$.

  If $I\in \Ni\setminus\Os$, then there is a unique cover of $I$ with respect to $\preceq$ (by Proposition \ref{prop:m3-non-idemp-one-cover}). Let $I^c$ be this unique cover. For any $J\in\Ni$, not comparable with $I$, then ${\uparrow}\{I, J\} = {\uparrow}\{I^c, J\}$ (by Lemma~\ref{lemma: poset with unique cover}). 
 Thus, $I\vee J$ exists if and only if $I^c\vee J$ exists, and if so, $I\vee J=I^c \vee J$. Also, by Proposition~\ref{prop:m3-non-idemp-one-cover}, $I+J=I^c+J$. We repeat this process for finitely many steps, thus producing an ascending chain $I\prec I^{c_1}\prec\dots\prec I^{c_n}$ of covers (more precisely, $I^{c_i}$ covers $I^{c_{i-1}}$, for every $i\in\{i,\dots n\}$) until either $I^{c_n}\in \Os$ or $J\preceq I^{c_n}$. 
 
 In the first case, by Lemma~\ref{lem:join-with-idempotent} we have $I^{c_n} \vee J = I^{c_n} + J =  I^{c_{n-1}} + J = \dots = I^{c_1} + J = I + J$, where all the right-side equalities follow by Proposition \ref{prop:m3-non-idemp-one-cover}.  By Lemma \ref{lemma: poset with unique cover}, applied several times, $I\vee J=I^{c_n}\vee J$, and we conclude $I\vee J=I+J$.
 
 In the latter case, $J\preceq I^{c_n}$. Obviously, $J\neq I^{c_n}$ (as otherwise $I\prec J$). Hence, by Lemma \ref{lem:m3-join-non-idempotent}, $I^{c_{n-1}}\vee J =I^{c_{n-1}} + J = \dots = I + J $, where the right-hand side equalities follows from Proposition \ref{prop:m3-non-idemp-one-cover} (where we are assuming that $I^{c_{i}}\in \Ni\setminus \Os$, for all $i\in\{1,\dots, n\}$ as otherwise we would be in the previous case). By Lemma~\ref{lemma: poset with unique cover}-(1), $I\vee J= I^{c_1}\vee J=\dots= I^{c_{n-1}}\vee J$. Putting all this together, $I+J=I^{c_{n-1}}\vee J=I^{c_n}\vee J=I\vee J$.

The description of $I\wedge J$ follows from Theorem \ref{th: Nation Th: 2.4}. Observe that, in the case $I$ and $J$ are comparable, that is, $I\preceq J$, which implies $I\subseteq J$, we have that $I\wedge J = \bigvee({\downarrow}I\cap{\downarrow}J) = \bigvee({\downarrow}I) = I$ (similarly for the case $J\preceq I$).
\end{proof}

For multiplicity four, it is no longer true that a normalized ideal $I$ that is not idempotent has a unique cover. Thus, we cannot take advantage of the same strategy used in the case of multiplicity three.

\begin{example}
 Let $S=\langle 4,7,9,10\rangle$, and let $I=(2,2,0)_{\mathcal{K}}$. Then, $I+I\neq I$ and the ideals covering $I$ are $(1,2,0)_{\mathcal{K}}$, $(2,1,0)_{\mathcal{K}}$, and $(0,2,0)_{\mathcal{K}}$.
\end{example}

\section{Quarks of the poset of normalized ideals}

Let $S$ be a numerical semigroup with multiplicity three. We already know that if $S$ is irreducible, then $\Ni$ has at most two quarks \cite[Theorem~5.21]{apery-icm}. Moreover, if $S$ is symmetric we know that the only quark of $S$ is $\{0,\operatorname{F}(S)\}+S$ \cite[Proposition~5.18]{apery-icm}, and if $S$ is pseudo-symmetric, then the quarks of $\Ni$ are $\{0,\operatorname{F}(S)\}+S$ and $\{0,\operatorname{F}(S)/2\}+S$ \cite[Proposition~5.20]{apery-icm}.

By \cite[Theorem~7]{multiplicity-3-4}, every numerical semigroup with multiplicity three is uniquely determined by its genus and its Frobenius number: $S=\langle 3,3\operatorname{g}(S)-\operatorname{F}(S),\operatorname{F}(S)+3\rangle$. 
By \cite[Remark~5.1]{apery-icm}, the genus of $S$ plus one equals the length of the largest ascending chain in $(\Ni,\preceq)$. If $S$ is symmetric, then its Frobenius number is twice the genus minus one (see for instance \cite[Corollary~4.5]{ns}), while if $S$ is pseudo-symmetric, the Frobenius number of $S$ equals twice its genus minus two \cite[Corollary~4.5]{ns}. This means that if $S$ has at most two quarks, we can fully recover $S$ from $(\Ni,\preceq)$.

Let us focus on the case $S$ is not irreducible. Recall that $S$ is irreducible if and only if the cardinality of $\operatorname{SG}(S)$ is at most one \cite[Corollary~4.38]{ns}. We also know that $\operatorname{SG}(S)\subseteq \operatorname{PF}(S)$, and that the type of $S$ is at most three \cite[Corollary~10.22]{ns}. Putting all this together, we deduce that in the case $S$ is not irreducible, the set of special gaps coincides with the set of pseudo-Frobenius elements. Let $\operatorname{SG}(S)=\operatorname{PF}(S)=\{f'<f\}$ (and so $f=\operatorname{F}(S)$). In particular, $S'=S\cup\{f'\}$ is a numerical semigroup, and so is $\overline{S}=S\cup \{f\}$ (notice that $f=\operatorname{F}(S)$ is always a special gap for any numerical semigroup different from $\mathbb{N}$).
 
In the next two lemmas we describe the set of normalized ideals of $S$ that are not ideals of $S'$ and of $\overline{S}$ (compare the first with \cite[Lemma~5.17]{apery-icm}).

\begin{lemma}\label{lem:ideals-not-containing-f}
   Let $S$ be a numerical semigroup with Frobenius number $f$. Then, $\overline{S}=S\cup\{f\}$ is a numerical semigroup, and for every $I\in \Ni$, $I\in \Ni\setminus\Ni[\overline{S}]$ if and only if $f\not\in I$. 
\end{lemma}
\begin{proof}
  We already know that $\overline{S}=S\cup\{f\}$ is a numerical semigroup. Notice that $\overline{S}=\{0,f\}+S$.

  If $f\in I$, then $\overline{S}=S\cup\{f\}\subseteq I$, and $I=I+0\subseteq I+\overline{S}= I+(\{0,f\}+S)=I+\{0,f\}=I\cup (f+I)$. By definition of Frobenius number, $f+I\subseteq f+\mathbb{N}\subseteq S\cup \{f\}=\overline{S}$. Hence, $I\cup(f+I)\subseteq I\cup\overline{S}=I$. Thus, $I\subseteq I+\overline{S}\subseteq I$, and so $I+\overline{S}=I$, which leads to $I\in \Ni[\overline{S}]$ Clearly, if $I\in \Ni[\overline{S}]$, then $f\in \overline{S}\subseteq I$. This proves the claim. 
\end{proof}

\begin{lemma}\label{lem:ideals-not-containing-f2}
Let $S=(k_1,k_2)_{\mathcal{K}}$ be a numerical semigroup with multiplicity three and $\operatorname{PF}(S)=\{f',f\}$ with $f'<f$ and $2f'\neq f$. Then, $S'=S\cup\{f'\}$ is a numerical semigroup. Moreover, the set of ideals of $\Ni\setminus \Ni[S']$ that contain $f'$ is 
\begin{enumerate}
  \item $\{(x_1,x_2)_{\mathcal{K}}\in \Ni : x_1+k_1=x_2, x_1\le k_1-1\}$ if $f'\equiv 1\pmod 3$,
  \item $\{(x_1,x_2)_{\mathcal{K}}\in \Ni : x_2+k_2+1=x_1, x_2\le k_2-1\}$ if $f'\equiv 2\pmod 3$.
\end{enumerate}
\end{lemma}
\begin{proof}
  We have seen already that under the standing hypothesis, $\operatorname{SG}(S)=\{f',f\}$, and so $S'=S\cup\{f'\}$ is a numerical semigroup.
  
  Notice that as $f'\not\in S$, $f'$ is not a multiple of three. Thus, $f'\bmod 3$ is either one or two. Suppose first that $f'\equiv 1 \pmod 3$, and let $k'=(f-1)/3$, that is, $f'=3k'+1$. It easily follows that $k'=k_1-1$ and that the Kunz coordinates of $S'$ are $(k_1-1,k_2)$.
  
  If $I=(x_1,x_2)_{\mathcal{K}}\in \Ni\setminus\Ni[S']$,  then by \eqref{eq:kunz-cond}, $x_1\le k_1$, $x_2\le k_2$, $x_1+k_1\ge x_2$, and $x_2+k_2+1\ge x_1$. The condition $f'\in I$ means that $x_1\le k_1-1$ by \eqref{eq:kunz-membership}. Therefore, as $I\not\in \Ni[S']$, and $(x_1,x_2)$ already fulfills the inequalities $x_1\le k_1-1$ and $x_2+k_2+1\ge x_1$, we deduce that $x_1+(k_1-1)<x_2$. If follows that, $x_1+k_1=x_2$. 
  
  For the other inclusion, if $I=(x_1,x_2)_{\mathcal{K}}\in \Ni$ is such that $x_1+k_1=x_2$ and $x_1\le k_1-1$ hold, then $f'\in I$ (because $x_1\le k_1-1=k'$) and $I\not\in \Ni[S']$, because $x_1+(k_1-1)=x_2-1<x_2$, see \eqref{eq:kunz-cond}.

  The other case follows analogously.
\end{proof}

Next, we see that if $S$ is not irreducible, then it has exactly three quarks.

\begin{proposition}\label{prop:mult-three-non-irreducible-three-quarks}
    Let $S$ be a non-irreducible numerical semigroup with multiplicity  three and Kunz coordinates $(k_1,k_2)$. 
\begin{itemize}
    \item If $k_1\le k_2$, then the quarks of $\Ni$ are $(k_1-1,k_2)_{\mathcal{K}}$, $(k_1,k_2-1)_{\mathcal{K}}$, and $(k_2-k_1,k_2)_{\mathcal{K}}$.
    \item If $k_1> k_2$, then the quarks of $\Ni$ are $(k_1-1,k_2)_{\mathcal{K}}$, $(k_1,k_2-1)_{\mathcal{K}}$, and $(k_1, k_1-k_2-1)_{\mathcal{K}}$.
\end{itemize}
\end{proposition}
\begin{proof}
    We know that under the standing hypothesis $\operatorname{SG}(S)=\operatorname{PF}(S)=\{f',f\}$ with $f'<f=\operatorname{F}(S)$. Also, $f'\neq f/2$ \cite[Corollary~4.16]{ns}, because this would mean that $S$ is pseudo-symmetric. 

    By Proposition~4.17 and Lemma~5.10 in \cite{apery-icm}, $S'=S\cup\{f'\}=\{0,f'\}+S$ and $\overline{S}=S\cup\{f\}=\{0,f\}+S$ are quarks in ${(\Ni,\preceq)}$. Since $\operatorname{SG}(S)=\{f',f\}$, $S'$ and $\overline{S}$ are numerical semigroups and so idempotent ideals of $S$ \cite[Proposition~5.14]{apery-icm}.  Thus, by \cite[Proposition~5.13]{apery-icm} the only idempotent quarks of $\Ni$ are $\{0,f'\}+S$ and $\{0,f\}+S$, and both are unitary extensions of $S$. 

    Observe also that from $\operatorname{PF}(S)=-\operatorname{m}(S)+\operatorname{Maximals}_{\le_S}(\operatorname{Ap}(S,3))$ \cite[Proposition~2.20]{ns} and the fact that $\operatorname{Ap}(S,3)=\{0,w_1,w_2\}$ with $w_1,w_2\in S^*$, we deduce that $\operatorname{Ap}(S,3)=\{0,f'+3,f+3\}$ (hence, $f\not\equiv f' \pmod 3$). In particular, by the definition of Kunz coordinates, $S$ has Kunz coordinates equal to $(k_1,k_2)$ with $k_1=(f'+3-1)/3$ and $k_2=(f+3-2)/2$ in the case $f'\equiv 1\pmod 3$ (and thus $f\equiv 2\pmod 3$), of $k_1=(f+3-1)/3+1$ and $k_2=(f'+3-2)/2+1$ in the case $f'\equiv 2\pmod 3$ (and thus $f\equiv 1\pmod 3$). Also, this means that $\{\{0,f'\}+S,\{0,f\}+S\}=\{(k_1-1,k_2)_{\mathcal{K}}, (k_1,k_2-1)_{\mathcal{K}}\}$. 

    With the above notation, we have $S=S'\setminus\{f'\}=\overline{S}\setminus\{f\}$. Observe that $\Ni[S']\cup \Ni[\overline{S}]\subseteq \Ni$. If $I\in \Ni[S']$, then $S'+I=I$ and so $S'\preceq I$; similarly, if $I\in \Ni[\overline{S}]$, then $\overline{S}\preceq I$. Thus, we are interested in the set $\Ni\setminus (\Ni[S']\cup \Ni[\overline{S}])$.

    In order to describe $\Ni\setminus (\Ni[S']\cup \Ni[\overline{S}])$, we distinguish two cases, depending on $(f' \bmod 3, f \bmod 3)$. Recall that $f\not\equiv f' \pmod 3$, and neither $f'$ nor $f$ is a multiple of three. 
    \begin{itemize}
        \item \emph{The case $(f' \bmod 3, f \bmod 3)=(1,2)$}. Set $k'=(f'-1)/3$ and $k=(f-2)/3$. As $f' < f$, we deduce that $k'\le k$. Also, $k'=k_1-1$ and $k=k_2-1$. 
        Let $(x_1,x_2)$ be the Kunz coordinates of an ideal $I\in \Ni\setminus (\Ni[S']\cup \Ni[\overline{S}])$. As $I\in \Ni$, by \eqref{eq:kunz-cond}, $x_2\le k_2$, and as $I\not\in \Ni[\overline{S}])$, by Lemma~\ref{lem:ideals-not-containing-f}, $f\not \in I$ and so $x_2\ge k+1=k_2$ by \eqref{eq:kunz-membership}. Thus, $x_2=k_2$. Also, by \eqref{eq:kunz-cond}, $x_1\le k_1$.
        
         If $f'\not\in I$, then by \eqref{eq:kunz-membership}, $k'=k_1-1<x_1$. Hence, $k_1-1<x_1\le k_1$, which forces $x_1=k_1$, and consequently $I=S$. 
        
        If $f'\in I$, we apply Lemma~\ref{lem:ideals-not-containing-f2}-(1) to obtain that $x_1+k_1=x_2$ and $x_1\le k_1-1$. As $x_2=k_2$, we deduce that $x_1=k_2-k_1$. It remains to show that for $x_1=k_2-k_1$, $x_1\le k_1-1$ holds. Observe that $x_1\le k_1-1$ if and only if $2k_1\ge k_2+1$. By \eqref{eq:kunz-cond} (or \eqref{eq:kunz-cond-idempotentes}) applied to $S$, we know that $2k_1\ge k_2$, so we need to ensure that $2k_1\neq k_2+1$. If $2k_1=k_2+1$, then $2f'=2(3(k_1-1)+1)=3(2k_1)-4=3(k_2+1)-4=3k_2-1=3(k_2-1)+2=f$, which contradicts the fact that $S$ is not pseudo-symmetric.
        
        This proves that the only ideals in $\Ni\setminus (\Ni[S']\cup \Ni[\overline{S}])$ are $S$ and $Q=(k_2-k_1,k_2)_{\mathcal{K}}$.

        \item \emph{The case $(f' \bmod 3, f \bmod 3)=(2,1)$}.  Set $k'=(f'-2)/3$ and $k=(f-1)/3$. As $f'<  f$, we deduce that $k'< k$. Also, $k'=k_2-1$ and $k=k_1-1$. 
        Let $(x_1,x_2)$ be the Kunz coordinates of an ideal $I\in \Ni\setminus (\Ni[S']\cup \Ni[\overline{S}])$. From $I\in \Ni$, by \eqref{eq:kunz-cond}, we obtain $x_1\le k_1$, and as $I\not\in \Ni[\overline{S}]$, by Lemma~\ref{lem:ideals-not-containing-f}, $f\not \in I$ and so $x_1\ge k+1=k_1$ by \eqref{eq:kunz-membership}. Thus, $x_1=k_1$. Also, by \eqref{eq:kunz-cond}, $x_2\le k_2$.

        If $f'\not\in I$, then by \eqref{eq:kunz-membership}, $k'=k_2-1<x_2$; whence, $k_2-1<x_2\le k_2$, which forces $x_2=k_2$, and consequently $I=S$. 

        If $f'\in I$, then by Lemma~\ref{lem:ideals-not-containing-f2}-(2) we deduce that $x_2+k_2+1=x_1$ and $x_2\le k_2-1$. The first equality yields $x_2=k_1-k_2-1$. Now, it remains to see that $k_1-k_2-1\le k_2-1$, or equivalently, $2k_2\ge k_1$. By \eqref{eq:kunz-cond-idempotentes} (or \eqref{eq:kunz-cond} applied to $(k_1,k_2)$), we know that $2k_2+1\ge k_1$, so we have to show that $2k_2+1=k_1$ cannot hold. If $k_1=2k_2+1$, then $2f'=2(3(k_2-1)+2)= 3(2k_2)-2=3(k_1-1)-2=f-3\not\in S$, which is impossible as we are assuming that $f'\in \operatorname{SG}(S)$ and so $2f'\in S$.
        
        Thus, the only ideals in $\Ni\setminus (\Ni[S']\cup \Ni[\overline{S}])$ are $S$ and $Q=(k_1,k_2-k_1-1)_{\mathcal{K}}$.         
    \end{itemize}
    Next, we prove that $Q$ is a quark in $\Ni$. We know that $\Ni=\Ni[S']\cup \Ni[\overline{S}]\cup \{S,Q\}$. Suppose on the contrary that there exists $I\neq S$ with $I\prec Q$. In particular, $I\neq Q$ and so either $I\in \Ni[S']$ or $I\in \Ni[\overline{S}]$. Recall that by Lemma~\ref{lem:oversemigroup-uparrow}, $\Ni[S']={\uparrow}S'$ and $\Ni[\overline{S}]={\uparrow}\overline{S}$. If $I\in \Ni[\overline{S}]$, then $S\preceq I\preceq Q$, and by Lemma~\ref{lem:oversemigroup-uparrow} we deduce that $Q\in \Ni[\overline{S}]$, a contradiction. Similarly, we obtain that $I\not\in \Ni[S']$, and this contradicts the fact that  either $I\in \Ni[S']$ or $I\in \Ni[\overline{S}]$.    

    Finally, observe that the case $k_1\le k_2$ corresponds with $(f'\bmod 3,f\bmod 3)=(1,2)$, while $k_1>k_2$ is equivalent to $(f'\bmod 3,f\bmod 3)=(2,1)$.
\end{proof}

Let $S$ be a numerical semigroup and let $I\in \Ni$. The depth of $I$ is the largest $k$ such that there exists a sequence $I_0,\dots, I_k\in \Ni$ with $I_0=I$, $I_k=\mathbb{N}$, $I_i\preceq I_{i+1}$, and $I_i\neq I_{i+1}$ for all $i\in\{0,\dots,k-1\}$. 

\begin{lemma}
    Let $S$ be a numerical semigroup and let $I=(x_1,x_2)_{\mathcal{K}}$ be an ideal of $S$. Then, the depth of $I$ is $x_1+x_2$.
\end{lemma}
\begin{proof}
  If $I$ is not idempotent, then by Proposition~\ref{prop:m3-non-idemp-one-cover}, we know that $I$ has a single cover, and it is either of the form $(x_1-1,x_2)_{\mathcal{K}}$ or of the form $(x_1,x_2-1)_{\mathcal{K}}$. Otherwise, $I$ is idempotent, and thus it is a numerical semigroup, and $I\cup\{\operatorname{F}(I)\}=\{0,\operatorname{F}(I)\}+I$ covers $I$. Every oversemigroup of $S$ has multiplicity three except $\langle 2,3\rangle$ and $\mathbb{N}$. Thus, in any case the Frobenius number of $I$ is not a multiple of three. Depending on the congruence class of $\operatorname{F}(I)$ modulo three, $I\cup\{\operatorname{F}(I)\}$ will be either $(x_1-1,x_2)_{\mathcal{K}}$ or $(x_1,x_2-1)_{\mathcal{K}}$. 

  In this way, we construct a sequence of ideals $I_0,\dots, I_k$, with $k=x_1+x_2$ such that $I_0=I$, $I_k=\mathbb{N}$ and $I_{i+1}$ covers $I_i$ for all $i$. 
  
  Observe also that if $I_0',\dots,I_t'$ is another sequence of ideals such that $I_0'=I$, $I_t'=\mathbb{N}$ and $I_{i+1}'$ covers $I_i$ for all $i$, then $I_i'\subsetneq I_{i+1}'$, which means that $|I_{i+1}'\setminus I_i'|\ge 1$. This forces $t$ to be upper bounded by $|\mathbb{N}\setminus I|$, which by \eqref{eq:kunz-coordinates-sum-general} equals $x_1+x_2$.
\end{proof}

If we apply this last result to $S$, we obtain that the depth of $S$ is $k_1+k_2$, recovering in this way \cite[Remark~5.1]{apery-icm}.

With all these ingredients at hand, we can now solve a particular instance of \cite[Question~6.2]{apery-icm}.

\begin{theorem}\label{thm:isom-poset-m3}
    Let $S$ and $T$ be two numerical semigroups, such that $\operatorname{m}(S)=3$. If $(\Ni,\preceq)$ and $(\Ni[T],\preceq)$ are order isomorphic, then $S=T$. 
\end{theorem}
\begin{proof}
    By \cite[Proposition~5.2]{apery-icm}, $S$ and $T$ have the same multiplicitity. Also, by \cite[Theorem~5.21]{apery-icm}, $S$ is irreducible if and only if $T$ is irreducible.

    If $S$ is irreducible, then by the discussion at the beginning of this section, $S$ is completely determined by the height of $(\Ni,\preceq)$ (the depth of $S$) and by the number of quarks. Thus, in this case $S$ must be equal to $T$.

    Suppose now that $S$ is not irreducible. Let $(k_1,k_2)$ be the Kunz coordinates of $S$. Let $g$ be the genus of $S$. We know that $S$ has three quarks. By \eqref{eq:kunz-coordinates-sum-general}, $k_1+k_2=g$. By the previous lemma two of the quarks of $S$ have the same depth, and it is equal to $g-1$, while the third quark, $Q$, has depth equal to $2k_2-k_1$ or $2k_1-k_2-1$. Let us call this depth $d$. In the first case, $g+d=3k_2$ and so $g+d\equiv 0 \pmod 3$, while in the second case $g+d=3k_1-1$, which means that $g+d\equiv 2\pmod 3$.

   Notice that $k_1$ and $k_2$ are solutions to one of these systems of equations: 
    \begin{itemize}
        \item $k_1+k_2=g$, $2k_2-k_1=d$,
        \item $k_1+k_2=g$, $2k_1-k_2-1=d$.
    \end{itemize}
    In the first case ($g+d\equiv 0 \pmod 3$), we obtain $k_1=(2g-d)/3$ and $k_2=(g+d)/3$; while in the second ($g+d\equiv 2\pmod 3$), $k_1=(g+d+1)/3$ and $k_2=(2g-d-1)/3$. 
    
    This proves that $k_1$ and $k_2$ (and thus $S$) are uniquely determined by $g$ and $d$, and both $g$ and $d$ can be read from the Hasse diagram of $(\Ni,\preceq)$.
\end{proof}

\begin{remark}
    Let us prove that in the non-irreducible case, the third quark $Q$ has depth smaller than the genus minus one of $S$ (which is the depth of the other two quarks), that is, $d<g-1$, with the notation of the proof of Theorem~\ref{thm:isom-poset-m3}. Notice that $d\le g-1$, since the depth of $S$ is $g$ and $Q$ covers $S$. If $d=g-1$, then in the first case, $2k_2-k_1= k_1+k_2-1$, and so $2k_1=k_2+1$. Then, $2f'=2(3(k_1-1)+1)=3(2k_1)-4=3(k_2+1)-4=3k_2-1=3(k_2-1)+2=f$, which contradicts the fact that $S$ is not pseudo-symmetric (as in the proof of Proposition~\ref{prop:mult-three-non-irreducible-three-quarks}). 
    In the second case, $2k_1-k_2-1=k_1+k_2-1$, and so $k_1=2k_2$ and consequently $2f'=2(3(k_2-1)+2)=3k_1-2=3(k_1-1)+1=f$, which is impossible as we are assuming that $S$ is not pseudo-symmetric.
\end{remark}

\begin{example}
    Let $S=\langle 4, 9, 14, 19\rangle$. Then, $\Ni[S]$ has three quarks two of them are idempotents. Their Kunz coordinates are $(1,3,4)$, $(2,2,4)$, and $(2,3,3)$. Thus, all of them have depth eight. This means that for multiplicity four the strategy employed in this section for non-irreducible numerical semigroups is no longer valid.
\end{example}

\section*{Acknowledgments}
The research was carried out thanks to the hospitality offered by the Institute of Mathematics (IMAG) of the University of Granada and the staff exchange program ERASMUS+ funded by the European Community.

S. Bonzio acknowledges the support by the Italian Ministry of Education, University and Research through the PRIN 2022 project DeKLA (``Developing Kleene Logics and their Applications'', project code: 2022SM4XC8) and the PRIN Pnrr project ``Quantum Models for Logic, Computation and Natural Processes (Qm4Np)'' (cod. P2022A52CR). He also acknowledges the Fondazione di Sardegna for the support received by the projects GOACT (grant number
F75F21001210007) and MAPS (grant number F73C23001550007), the University of Cagliari for the support by the StartUp project ``GraphNet''. Finally, he gratefully acknowledges also the support of the INDAM GNSAGA (Gruppo Nazionale per le Strutture Algebriche, Geometriche e loro Applicazioni). \\
\noindent
P. A. García-Sánchez is partially supported by the grant number ProyExcel\_00868 (Proyecto de Excelencia de la Junta de Andalucía) and by the Junta de Andaluc\'ia Grant Number FQM--343. He also acknowledges financial support from the grant PID2022-138906NB-C21 funded by MICIU/AEI/10.13039/\bignumber{501100011033} and by ERDF ``A way of making Europe'', and from the Spanish Ministry of Science and Innovation (MICINN), through the ``Severo Ochoa and María de Maeztu Programme for Centres and Unities of Excellence'' (CEX2020-001105-M).


\begin{thebibliography}{9}

\bibitem{ns-app} A. Assi, M. D'Anna, P. A. García-Sánchez, Numerical semigroups and applications, Second edition, RSME Springer series \textbf{3}, Springer, Switzerland, 2020.


\bibitem{b-k} V. Barucci, F. Khouja, \textit{On the class semigroup of a numerical semigroup}, Semigroup Forum, \textbf{92} (2016), 377--392. 

\bibitem{laura} L. Casabella, \textit{On the class semigroup of numerical semigroups and semigroup rings}, Diploma thesis, Scuola Superiore di Catania (2022), unpublished.

\bibitem{apery-icm} L. Casabella, M. D'Anna, P. A. García-Sánchez, Apéry sets and the ideal class monoid of a numerical semigroup, Mediterr. J. Math. 21 (2024), Article No. 7 (28 pages). 

\bibitem{numericalsgps} M. Delgado, P. A. García-Sánchez, J. Morais, NumericalSgps, A package  for  numerical  semigroups,  Version 1.3.1 dev (2023), Refereed GAP package, \url{https://gap-packages.github.io/numericalsgps}.

\bibitem{gap} The GAP-Group, GAP -- Groups, Algorithms, and Programming, Version 4.12.2, (2022), \url{https://www.gap-system.org}.

\bibitem{isom-icm} P. A. Garc\'ia-S\'anchez, \textit{The isomorphism problem for ideal class monoids of numerical semigroups}, Semigroup Forum \textbf{108} (2024), 365--376.

\bibitem{Nation} J. N. Nation, \textit{Notes on Lattice Theory}, Lecture Notes, 2017.

\bibitem{multiplicity-3-4} J. C. Rosales, Numerical semigroups with multiplicity three and four, Semigroup Forum \textbf{71} (2005), 323--331.

\bibitem{ns} J. C. Rosales, P. A. Garc\'ia-S\'anchez, Numerical semigroups, Developments in Mathematics, \textbf{20}, Springer, New York, 2009.

\end{thebibliography}
\end{document}